\documentclass[a4paper,11pt,oneside]{article}
\usepackage{graphicx}
\usepackage{amscd}
\usepackage{amsmath}
\usepackage{caption}
\usepackage{amsfonts}
\usepackage{amssymb}
\usepackage{amsthm}
\usepackage{mathrsfs}
\usepackage{multicol}
\usepackage{color}
\usepackage[english]{babel}
\usepackage[T1]{fontenc}
\usepackage{textcomp}
\usepackage[utf8]{inputenc}
\usepackage{enumitem}
\usepackage{indentfirst}
\usepackage{hyperref}
\usepackage[title,titletoc]{appendix}

\usepackage{romannum}

\makeatletter
\renewcommand*{\eqref}[1]{%
  \hyperref[{#1}]{\textup{\tagform@{\ref*{#1}}}}%
}
\makeatother

\newcommand{\N}{\mathbb N}

\newcommand{\ZZ}{\mathbb{Z}}
\newcommand{\RR}{\mathbb{R}}
\newcommand{\CC}{\mathbb{C}}

\newcommand{\DD}{\Delta}
\newcommand{\MM}{M}

\def\({\begin{eqnarray}}
\def\){\end{eqnarray}}
\def\[{\begin{eqnarray*}}
\def\]{\end{eqnarray*}}

\def\d{\mathrm{d}}
\def\dx{d}

\def\RI{\mbox{\Romannum{1}}}
\def\RII{\mbox{\Romannum{2}}}

\begin{document}
\pagenumbering{arabic}

	\theoremstyle{plain}
	\newtheorem{thm}{Theorem}[section]
	\newtheorem{cor}[thm]{Corollary}
	\newtheorem{lem}[thm]{Lemma}
	\newtheorem{prop}[thm]{Proposition}
	\theoremstyle{definition} \newtheorem{defn}{Definition}[section]
	
	\newtheorem{oss}[thm]{Remark}
	\newtheorem{ex}{Example}[section]

	\title{Consensus and flocking with transmission and reaction delays}
	\author{E. Continelli\footnote{Dipartimento di Matematica \textgravedbl Tullio Levi-Civita\textacutedbl, Universit\`{a} degli Studi di Padova, Via Trieste 63, 35121 Padova, Italy. Email: {\it elisa.continelli@math.unipd.it.}}, J. Haskovec\footnote{Mathematical and Computer Sciences and Engineering Division, King Abdullah University of Science and Technology, Thuwal 23955-6900, Kingdom of Saudi Arabia. Email: {\it jan.haskovec@kaust.edu.sa}}, C. Pignotti\footnote{Dipartimento di Ingegneria e Scienze dell'Informazione e Matematica, Universit\`{a} degli Studi dell'Aquila,	Via Vetoio, Loc. Coppito, 67100 L'Aquila, Italy. Email: {\it cristina.pignotti@univaq.it.}}}
	\maketitle
	

\maketitle

\begin{abstract}
We investigate consensus formation and flocking behavior in multi-agent systems subject to two distinct types of delays:
a transmission delay accounting for information exchange between agents,
and a reaction delay representing the processing time before agents adjust their states.
For a simplified linear two-agent system, we provide explicit insight into how these delays affect asymptotic stability.
We then derive sufficient conditions for asymptotic consensus and flocking in the general multi-agent setting with a nonlinear, globally positive influence function.
These conditions require the delays to be sufficiently small relatively to the initial data and the decay rate of the influence function.
The analysis is based on a Lyapunov functional approach combined with a Halanay-type inequality.
Our results establish rigorous conditions under which collective behavior emerges
in delayed multi-agent systems where both communication and reaction lags are non-negligible,
with applications to biological, social, and engineered systems.
\end{abstract}
\vspace{2mm}

\vspace{2mm}

\vspace{2mm}

\section{Introduction and main results}\label{sec:Intro}
We consider a group of $N\in\N$, $N\geq 2$, agents with time-dependent opinions represented by the vectors $x_i=x_i(t)\in\mathbb{R}^d$ for $i \in [N]$,
with $d\in\N$ the space dimension and the notation $[N]:= \{1,2,\dots,N\}$.
The agents seek consensus by adapting their opinions to the ones of their peers through a superposition of pairwise interactions.
The established Hegselmann-Krause model \cite{HK} assumes that the influence of the agent $i$ on agent $j$'s opinion
depends on the difference between their opinions $|x_i-x_j|$ through a nonlinear \emph{influence function} $\psi: \mathbb{R}^+ \to \mathbb{R}^+$,
where here and in the sequel $\mathbb{R}^+$ denotes the set of nonnegative real numbers.

For applications in biological or socio-economical systems it is natural to assume
that the transmission of information between the agents takes a non-negligible amount of time \cite{Smith, E3B}.
This means that agent $i$ at time $t\geq 0$ receives the opinion of agent $j$ in the form $x_j(t-\kappa)$.
For simplicity, we assume that the transmission delay $\kappa>0$ is a global constant,
taking the same value for all pairs of agents.
Moreover, we assume that the agents need a certain amount of time $\sigma\geq 0$ to process the received information
and initiate the appropriate reaction. Again, for simplicity we assume $\sigma\geq 0$ to be a global constant.
Then, the reaction of agent $i$ at time $t\geq 0$ is based on the states $x_j(t-\sigma-\kappa)$ of all other agents $j\neq i$,
and on its own state $x_i(t-\sigma)$.
Introducing the above assumptions into the Hegselmann-Krause model \cite{HK} and denoting $\tau:=\sigma+\kappa$,
we obtain the system
\begin{equation}\label{hkp}
   \frac{\d}{\d t}x_{i}(t) = \underset{j:j\neq i}{\sum}a_{ij}(t)(x_{j}(t-\tau)-x_{i}(t-\sigma)) \qquad \mbox{for } t > 0,\; i\in[N],
\end{equation}
where $\tau\geq \sigma$, and the symbol $\underset{j:j\neq i}\sum$ denotes summation over all $j\in[N] \setminus\{i\}$.
The communication rates $a_{ij} > 0$, $i,j\in[N]$, are defined as
\begin{equation}\label{commrates}
   a_{ij}(t):=\frac{1}{N-1}\psi(\lvert x_i(t-\sigma)-x_j(t-\tau)\rvert) \qquad \mbox{for } t\geq 0,\;  i,j\in[N],
\end{equation}
with the influence function $\psi: \mathbb{R}^+ \to \mathbb{R}^+$ assumed to be globally positive, continuous and nonincreasing.
The uniform positivity is obviously necessary for reaching global consensus (otherwise, clustering would take place in general, see, e.g., \cite{JM}).
The monotonicity is natural from the modeling point of view, since it can be expected that
agents with less difference in their opinions would influence each other more strongly.
However, it is not strictly necessary for our analysis to apply.
For the possible case of nonmonotone influence function $\psi$,
it can be replaced by its nonincreasing rearrangement
$\Psi(u) := \min_{s\in[0,u]} \psi(s)$  for $u\geq 0$.
Then all our results remain valid.
Moreover, without loss of generality (after an eventual rescaling) we assume that $\lVert\psi\rVert_{L^\infty(\mathbb{R}^+)} \leq 1$.
We then have
\(  \label{bound:a}
   a_{ij}(t) \leq \frac{1}{N-1} \qquad\mbox{for } t\geq 0, \; i\in[N].
\)
We prescribe the initial condition
\begin{equation}\label{incond}
	x_{i}(t)=x^{0}_{i}(t) \qquad \mbox{for } t\in [-\tau ,0],\; i\in[N],
\end{equation}
with continuous trajectories $x^0_i\in C([-\tau,0],\mathbb{R}^d)$.

Existence of unique global solutions for \eqref{hkp}--\eqref{incond} is obtained by the standard method of steps, see, e.g., \cite{Halanay, Hale, Smith}.
Given a global solution $\{x_i\}_{i\in[N]}$, we define the diameter $d_x = d_x(t)$,
\(  \label{def:dx}
   d_x(t) := \max_{i,j\in[N]}\lvert x_i(t)-x_j(t)\rvert \qquad \mbox{for } t\geq 0.
\)
\begin{defn}\label{def:consensus}
We say that the solution $\{x_i\}_{i\in[N]}$ of \eqref{hkp}--\eqref{incond} converges to a global consensus if
\[
    \lim_{t\to +\infty} d_x(t)=0.
\]
\end{defn}

To formulate our main result, we introduce the notation
\(   \label{def:DD0}
   \DD^0_x := \max_{i,j \in [N]} \max_{s, t\in [-\tau,0]}\lvert x^0_i(s) - x^0_j(t) \rvert.
\)
Moreover, for $k\in\N$, define
\(   \label{def:ZK}
     Z^k_\sigma := \frac{1}{2\sqrt{\sigma(1+\sigma)}} \left[ \left( (1+\sigma) + \sqrt{\sigma(1+\sigma)} \right)^{k+1} - \left( (1+\sigma) - \sqrt{\sigma(1+\sigma)} \right)^{k+1} \right].
\)
We remark that $Z^k_\sigma$ is a polynomial expression in $\sigma$, for any $k\in\N$.

\begin{thm}\label{thm:consensus}
Let $0 \leq \sigma \leq \tau$ and denote $K := \lceil 2\tau/\sigma \rceil$, i.e., the smallest integer such that $K\sigma \geq 2\tau$.
Assume that there exists $\beta>0$ such that
\(   \label{cond:beta}
   4\tau \leq \beta\left(2e^{-2\tau}-1\right),
\)
and
\(   \label{condition}
   4\tau+\beta(1-e^{-2\tau}) < \psi\left( \left(1 + \tau-\sigma + \beta^{-1}e^{2\tau} \right) \mathcal{Z}^K_\sigma \DD^0_x\right),
\)
with $\DD^0_x$ given by \eqref{def:DD0} and
\(   \label{def:Z}
    \mathcal{Z}^K_\sigma := Z^K_\sigma + Z^{K-1}_\sigma \beta \left( 1 - (1+2\tau) e^{-2\tau} \right),
\)
where $Z^K_\sigma$ and $Z_\sigma^{K-1}$ are given by \eqref{def:ZK}.
\\Then, the solution of \eqref{hkp}--\eqref{incond} reaches asymptotic consensus
in the sense of Definition~\ref{def:consensus}, with the diameter $d_x=d_x(t)$ 
decaying exponentially in time.
\end{thm}

Let us note that assumptions \eqref{cond:beta}--\eqref{condition} represent, for a fixed initial datum $x^0\in C([-\tau,0])^{Nd}$, a smallness
condition on $\tau$.
Indeed, a simple calculation reveals that \eqref{cond:beta} can only hold if $\tau<\frac{\ln 2}{2}$.
Moreover, for the special case $\psi:=1$, for which \eqref{hkp} represents a linear consensus problem,
it is meaningful to choose the smallest possible $\beta>0$ in \eqref{cond:beta}, which is
\[
   \beta := \frac{4\tau}{2e^{-2\tau}-1}.
\]
 Condition \eqref{condition} then becomes
\[
    \frac{4\tau e^{-2\tau}}{2e^{-2\tau}-1} < 1,
\]
which can be resolved to $\tau < \frac12 \left( 1 - W\left(\frac{e}{2}\right)\right) \approx 0.157$,
where $W(\cdot)$ is the Lambert $W$-function,
i.e., $W(s) e^{W(s)} = s$ for $s\geq 0$.
Let us note that in this case the flocking condition
does not depend on the initial datum \eqref{incond}, as can be expected for a linear problem
due to its scale invariance.
For a general influence function $\psi$ and a fixed initial datum $x^0$, we note that \eqref{cond:beta}--\eqref{condition} is
satisfied for $\sigma=\tau=0$ with an arbitrary $\beta>0$ if $\psi$ is globally positive, which we assume.
Then, by continuity, it is also satisfiable for small enough values of $0\leq\sigma\leq\tau$.

The proof of Theorem~\ref{thm:consensus} will be carried out in Section~\ref{sec:consensus},
based on construction of a Lyapunov functional.
The literature on asymptotic consensus under the presence of delay in discrete agent models
is extensive. We refer to \cite{CPP, CCP:25, CP:25, CC:23, CC:24, H:BullLon:21, H:SIADS:21, H:ProcAMS:21, H:MMAS:22, H:PDEA:25, Lu, PP:23, Seuret, E6, Zhou-Li}
and references therein.
However, to our best knowledge, no results exist for the case of transmission-reaction delay,
i.e., $0<\sigma<\tau$, which we study in this paper.
Let us note that the analysis is complicated by the fact that \eqref{hkp} does not preserve
the mean $\sum_{i\in[N]} x_i(t)$ when $\sigma\neq\tau$.
Moreover, no a-priori bound on $\max_{i\in[N]} |x_i(t)|$ is available.

In the second part of this paper we extend our analysis to the second-order version of \eqref{hkp},
i.e., the Cucker-Smale model for the description of flocking phenomena \cite{CS1, CS2}.
Here the state of the agents is given by their positions $x_i=x_i(t)\in\mathbb{R}^d$ and velocities $v_i=v_i(t)\in\mathbb{R}^d$
in the $d$-dimensional space, subject to the dynamics
\(
    \label{CS1}
   \frac{\d}{\d t} x_i(t) &=& v_i(t),\\
   \label{CS2}
   \frac{\d}{\d t} v_i(t) &=& \underset{j:j\neq i}{\sum}a_{ij}(t)(v_{j}(t-\tau)-v_{i}(t-\sigma)).
\)
The communication rates $a_{ij} > 0$, $i,j\in[N]$, are defined in \eqref{commrates}
and the influence function $\psi: \mathbb{R}^+ \to \mathbb{R}^+$ is again
assumed to be globally positive, continuous and nonincreasing.
Moreover, without loss of generality we impose $\lVert\psi\rVert_{L^\infty(\mathbb{R}^+)} \leq 1$.
The system \eqref{CS1}--\eqref{CS2} is equipped with the initial datum
\(  \label{CS:incond}
	x_{i}(t)=x^{0}_{i}(t), \; v_i(t)=v^0_i(t) \qquad \mbox{for } t\in [-\tau ,0],\; i\in[N],
\)
with trajectories $x^0_i\in C^1([-\tau,0],\mathbb{R}^d)$, $v^0_i\in C([-\tau,0],\mathbb{R}^d)$ such that
\(   \label{CS:incomp}
   \frac{\d}{\d t} x_i(t) = v_i(t) \qquad \mbox{for } t\in [-\tau ,0],\; i\in[N].
\)

\begin{defn}\label{def:flocking}
We say that the solution $\{x_i, v_i\}_{i\in[N]}$ of \eqref{CS1}--\eqref{CS:incond} exhibits asymptotic flocking if
\[
   \sup_{t\geq0} d_x(t) < +\infty\qquad \mbox{and}\qquad \lim_{t\to +\infty} d_v(t)=0,
\]
with the position diameter $d_x=d_x(t)$ given by \eqref{def:dx} and the
velocity diameter $d_v=d_v(t)$ defined as
\(  \label{def:dv}
   d_v(t) := \max_{i,j\in[N]}\lvert v_i(t)-v_j(t)\rvert \qquad \mbox{for } t\geq 0.
\)
\end{defn}

To formulate our main result, we introduce the notation
\(   \label{def:DDv0}
   \DD^0_v := \max_{i,j \in [N]} \max_{s, t\in [-\tau,0]}\lvert v^0_i(s) - v^0_j(t) \rvert.
\)

\begin{thm}\label{thm:flocking}
Let $0 \leq \sigma \leq \tau$ and denote $K := \lceil 2\tau/\sigma \rceil$, i.e., the smallest integer such that $K\sigma \geq 2\tau$.
Assume that there exists $\beta>0$ such that
\(   \label{CS:cond:beta}
   4\tau \leq \beta\left(2e^{-2\tau}-1\right),
\)
and $C>0$ such that
\( \label{CS:condition}\begin{split}
	e^{C\tau}(4\tau e^{C\tau} &+ \beta(1-e^{-2\tau})) + C \\
	&\leq \psi\left( \DD^0_x + \mathcal{W}_\sigma^K \DD^0_v + \frac{e^{C\tau}}{C} \left(1 + \beta^{-1}e^{2\tau + C\sigma} + e^{C\tau} (\tau-\sigma)\right) \mathcal{Z}^K_\sigma \DD^0_v \right),
\end{split} \)
\noindent with $\mathcal{Z}^K_\sigma$ defined in \eqref{def:Z} and
\(   \label{def:W}
    \mathcal{W}_\sigma^K := Z^K_\sigma - (1+ \sigma)\left( Z^{K-1}_\sigma + 1\right) ,
\)
where $Z^K_\sigma$ and $Z^{K-1}_\sigma$ are given by \eqref{def:ZK}.
\\Then, the solution of \eqref{CS1}--\eqref{CS:incond} exhibits asymptotic flocking
in the sense of Definition~\ref{def:flocking}, with the diameter $d_v=d_v(t)$ 
decaying exponentially in time with rate $C$.
\end{thm}

Let us note that, for the case when $\sigma=\tau=0$, \eqref{CS:condition} reduces to
\[
   C \leq \psi\left( \DD^0_x + \frac{1}{C} (1 + \beta^{-1})  \right).
\]
Since \eqref{CS:cond:beta} obviously holds with any $\beta>0$,  \eqref{CS:condition} reduces further to
\[
   C < \psi\left( \DD^0_x + \frac{1}{C} \right),
\]
i.e., we need $\psi$ decaying slower than $1/s$ as $s\to+\infty$. In particular, any $\psi(s) := \frac{1}{(1+s)^\alpha}$ with $\alpha<1$ verifies the condition.
Let us remark that for the classical Cucker-Smale model (without delays), the sharp condition for unconditional flocking is that $\int^{+\infty} \psi(s) \d s = +\infty$.
Therefore, for $\sigma=\tau=0$, \eqref{CS:condition} is almost optimal, excluding only the case when $\psi(s) \approx 1/s$ as $s\to+\infty$.

The proof of Theorem~\ref{thm:flocking} will be given in Section~\ref{sec:flocking}.
Again, there exists a substantial amount of literature establishing sufficient
conditions for asymptotic flocking in the Cucker-Smale model under the presence of delay,
see, e.g., \cite{CL, CCP:25, C:23,  DH, EHS:SIAM:16, H:SIADS:21, H:JMAA:22, H:KRM:23, HM:KRM:20, HM:MBE:20, Liu-Wu, PT:18, RC:22, E6}.
However, to our best knowledge, no flocking results for the transmission-reaction case $\sigma\neq\tau$
appeared so far.

The paper is organized as follows.
In Section~\ref{sec:toy} we consider the maximally simplified version of \eqref{hkp} with two agents
and constant communication rates. We study the asymptotic stability of the trivial steady state
of the resulting equation.
In Section~\ref{sec:consensus} we develop the analysis of the asymptotic
consensus for system \eqref{hkp}--\eqref{incond} and provide proof of Theorem~\ref{thm:consensus}.
Finally, in Section~\ref{sec:flocking} we extend our results to establish asymptotic
flocking for system \eqref{CS1}--\eqref{CS:incond}, in particular, to prove Theorem~\ref{thm:flocking}.

\section{Toy model: $N=2$}\label{sec:toy}
\setcounter{equation}{0}
To gain insights into the dynamics of the model,
we consider \eqref{hkp} with two agents, $N=2$,
and with constant communication rates $a_{12} = a_{21} = 1$.
Denoting $u(t) := x_1(t)-x_2(t)$, we obtain the scalar delay differential equation
\(   \label{eq:u}
   \dot u(t) = - u(t-\tau) - u(t-\sigma).
\)
Let us investigate the asymptotic stability of the trivial steady state $u\equiv 0$
for \eqref{eq:u}.
We make the ansatz $u(t) = e^{\lambda t}$ with $\lambda\in\CC$,
which leads to the characteristic equation
\(   \label{eq:char}
   \lambda + e^{-\lambda\tau} + e^{-\lambda\sigma} = 0.
\)
As we look for values of $\tau$, $\sigma>0$ where the roots
cross the imaginary axis, we set $\lambda = \omega i$, with nonzero $\omega\in\RR$.
Taking the real and imaginary parts of \eqref{eq:char} yields
\(
   \label{eq:cos}
   \cos(\omega\tau) + \cos(\omega\sigma) &=& 0,\\
   \sin(\omega\tau) + \sin(\omega\sigma) &=& \omega,
   \label{eq:sin}
\)
which we further rewrite as
\(  \label{eq:ch1}
   2\cos\left( \omega\frac{\tau+\sigma}{2}\right) \cos\left( \omega\frac{\tau-\sigma}{2}\right) &=& 0, \\
   \label{eq:ch2}
   2\sin\left( \omega\frac{\tau+\sigma}{2}\right) \cos\left( \omega\frac{\tau-\sigma}{2}\right) &=& \omega.
\)
As we impose $\omega\neq 0$, \eqref{eq:ch2} implies $\cos\left( \omega\frac{\tau-\sigma}{2}\right) \neq 0$,
and \eqref{eq:ch1}, in turn,  $\cos\left( \omega\frac{\tau+\sigma}{2}\right) = 0$.
Consequently,
\(  \label{eq:m}
   (\tau+\sigma)\omega = (2m+1)\pi, \qquad\mbox{for some } m\in\ZZ.
\)
Then $\sin\left( \omega\frac{\tau+\sigma}{2}\right) = (-1)^m$,
and inserting this into \eqref{eq:ch2} yields
\(  \label{eq:cs}
   2 \cos\left( \omega\frac{\tau-\sigma}{2}\right) = (-1)^m \omega.
\)
Moreover, \eqref{eq:m} gives $(\tau-\sigma)\omega = (2m+1)\pi - 2\omega\sigma$,
and using the identity
\[
   \cos\left( (2m+1)\frac{\pi}2 - \omega\sigma\right) = \sin\left( (2m+1)\frac{\pi}2 \right) \sin(\omega\sigma) = (-1)^m \sin(\omega\sigma)
\]
in \eqref{eq:cs} finally yields $2\sin(\omega\sigma) = \omega$.
This is solvable for $\omega\neq 0$ if and only if $\sigma>1/2$.
Inserting into \eqref{eq:sin} gives $2\sin(\omega\tau) = \omega$
and the solvability condition $\tau>1/2$.
We conclude that whenever $\min\{\sigma,\tau\}\leq 1/2$,
 the characteristic equation \eqref{eq:char}
does not admit purely imaginary roots,
and, therefore, Hopf bifurcations cannot occur.
Consequently, the trivial solution $u\equiv 0$ of \eqref{eq:u}
is asymptotically stable.

We calculated the pure imaginary roots of \eqref{eq:char} numerically in Fig.~\ref{fig:hopf}.
The area below the curve marked as $m=0$ is the set of asymptotic stability
of the zero solution of \eqref{eq:u}. We observe that the curve indeed matches
the region $\min\{\sigma,\tau\}\leq 1/2$ closely. The black dot indicates the point
$\sigma=\tau=\pi/4$, which is the well known Hopf bifurcation point
for the equation $\dot u(t) = -2u(t-\pi/4)$.

\begin{figure}
\begin{center}
	\includegraphics[width=.55\textwidth]{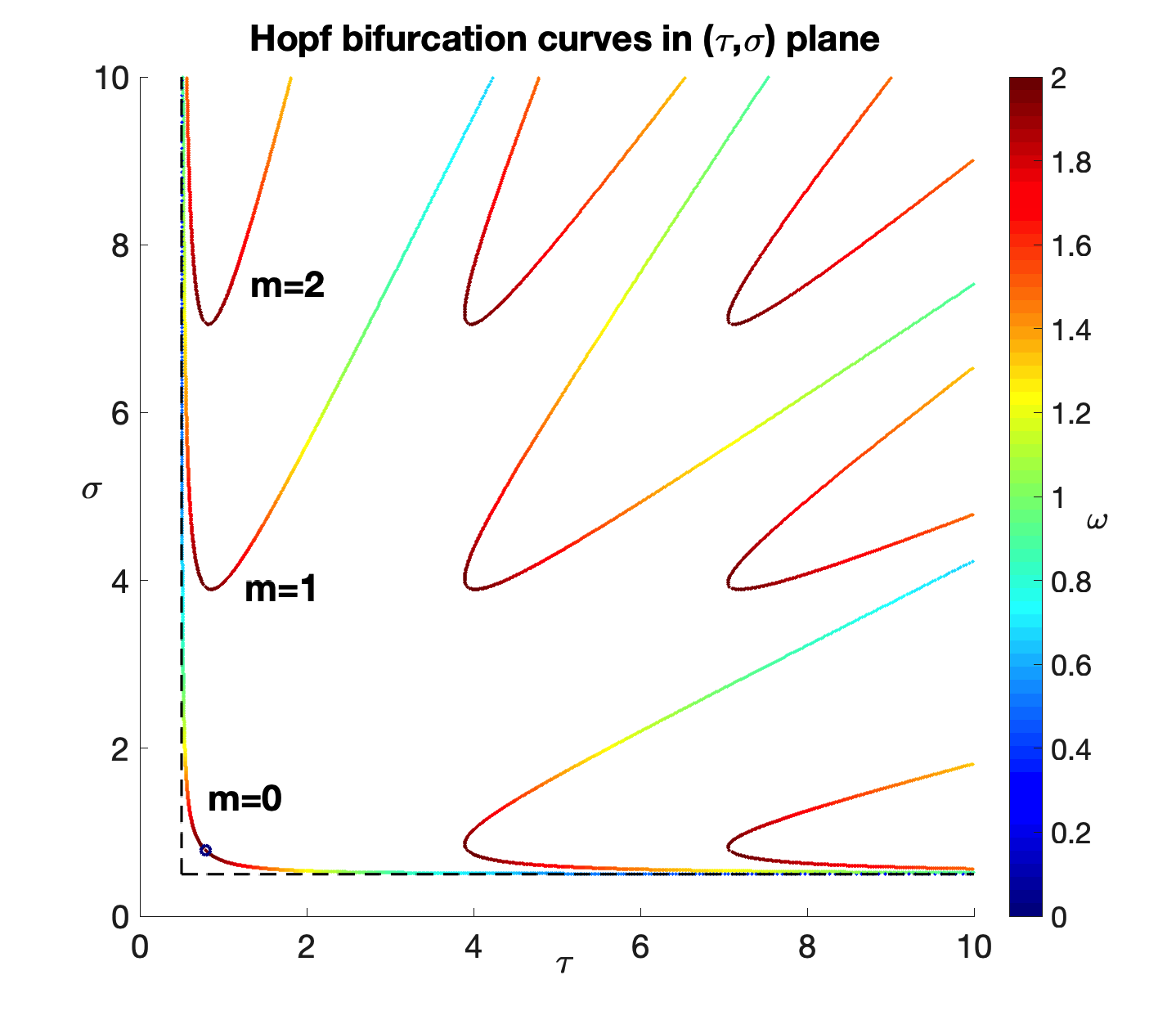}
\end{center}
\caption{Pure imaginary roots $\lambda = \omega i$ for the characteristic equation \eqref{eq:char}
colored by the value of $\omega\in (0,2]$,
obtained by numerical solution of the system \eqref{eq:cos}--\eqref{eq:sin}.
The black dashed lines demarcate the region where $\min\{\sigma,\tau\}\leq 1/2$.
The black circle marks the point $\sigma=\tau=\pi/4$.
}
\label{fig:hopf}
\end{figure}

\section{Asymptotic consensus in system \eqref{hkp}--\eqref{incond}}\label{sec:consensus}
\setcounter{equation}{0}

Throughout this section we denote $\{x_i\}_{i\in[N]}$
the global solution of \eqref{hkp}--\eqref{bound:a} with the initial datum \eqref{incond}. 

\subsection{Auxiliary results}\label{subsec:aux}
\begin{lem}\label{lem:aux1}
	For all $t\geq \tau$ it holds
	\begin{equation}\label{estonderiv}
		\max_{i\in[N]}\lvert \dot{x}_i(t)\rvert\leq d_x(t-\tau)+\int_{t-\tau}^{t-\sigma}\max_{i\in[N]} \lvert \dot{x}_i(s)\rvert \,\d s,
	\end{equation}
	with the diameter $d_x=d_x(t)$ defined in \eqref{def:dx}.
	\end{lem}
\begin{proof}
        From \eqref{hkp} we have, for any $i\in[N]$ and $t\geq \tau$,
	$$\begin{array}{l}
		\vspace{0.3cm}\displaystyle{\lvert\dot{x}_i(t)\rvert
                  \leq \underset{j:j\neq i}{\sum}a_{ij}(t)\lvert x_{j}(t-\tau)-x_{i}(t-\sigma)\rvert}\\
		\vspace{0.3cm}\displaystyle{\hspace{1.2cm}\leq \underset{j:j\neq i}{\sum}a_{ij}(t)\lvert x_{j}(t-\tau)-x_{i}(t-\tau)\rvert+ \underset{j:j\neq i}{\sum}a_{ij}(t)\lvert x_{i}(t-\tau)-x_{i}(t-\sigma)\rvert}\\
		\vspace{0.3cm}\displaystyle{\hspace{1.2cm}= \underset{j:j\neq i}{\sum}a_{ij}(t)\lvert x_{j}(t-\tau)-x_{i}(t-\tau)\rvert+ \lvert x_{i}(t-\tau)-x_{i}(t-\sigma)\rvert\underset{j:j\neq i}{\sum}a_{ij}(t)}\\
		\displaystyle{\hspace{1.2cm}
		\leq \underset{j:j\neq i}{\sum}a_{ij}\lvert x_{j}(t-\tau)-x_{i}(t-\tau)\rvert+ \lvert x_{i}(t-\tau)-x_{i}(t-\sigma)\rvert,}
	\end{array}$$
	where we used \eqref{bound:a} in the last inequality.
        By definition \eqref{def:dx} of $d_x=d_x(t)$ we readily have
                 $$\lvert x_j(t-\tau)-x_i(t-\tau)\rvert
                 \leq d_x(t-\tau).$$
     Therefore, using \eqref{bound:a} again,
	\begin{equation}\label{firstestimate}
		\lvert\dot{x}_i(t)\rvert
		\leq d_x(t-\tau)+\lvert x_{i}(t-\tau)-x_{i}(t-\sigma)\rvert.
	\end{equation}
	Since $t-\sigma \geq t-\tau \geq 0$, we estimate
	\[
	   \lvert x_{i}(t-\tau)-x_{i}(t-\sigma)\rvert &\leq& \int_{t-\tau}^{t-\sigma}\lvert \dot{x}_i(s)\rvert \,\d s  \\
               &\leq& \int_{t-\tau}^{t-\sigma}\max_{j\in[N]}\lvert \dot{x}_j(s)\rvert \,\d s.
          \]
	Plugging into \eqref{firstestimate}, we obtain
	\(  \label{est:dot_x_i}
	     \lvert\dot{x}_i(t)\rvert\leq d_x(t-\tau)+\int_{t-\tau}^{t-\sigma}\max_{j\in[N]}\lvert \dot{x}_j(s)\rvert \,\d s,
	\)
	and taking the maximum over $i\in[N]$, we finally conclude \eqref{estonderiv}.
\end{proof}

\begin{lem} \label{lem:aux2}
	For almost all $t>2\tau$ we have
	\begin{equation}\label{estwithextraterm}
		\begin{array}{l}
			\vspace{0.3cm}\displaystyle{\frac{\d}{\d t}d_x(t)\leq 2\int_{t-\tau}^{t}d_x(s-\tau) \,\d s+2\int_{t-\sigma}^{t}d_x(s-\tau) \,\d s}\\
			 \displaystyle{\hspace{1.3cm} + 2\int_{t-\tau}^{t}\int_{s-\tau}^{s}\max_{l\in[N]}\lvert \dot{x}_l(r)\rvert \,\d r \d s
			   + 2\int_{t-\sigma}^{t}\int_{s-\tau}^{s}\max_{l\in[N]}\lvert \dot{x}_l(r)\rvert \,\d r \d s - 
			N \underline{a}(t) d_x(t),}
		\end{array}
	\end{equation}
	where \begin{equation}\label{underlinea}
		\underline{a}(t):=\min_{i,j\in[N]} a_{ij}(t),
	\end{equation}
	and $a_{ij} = a_{ij}(t)$ given by \eqref{commrates}.
\end{lem}

\begin{proof}
	Due to the continuity of the trajectories there exists an at most countable system of open, mutually disjoint intervals $\{I_n\}_{n\in \mathbb{N}}$ such that
	$$\bigcup_{n\in \mathbb{N}}\overline{I_n}= [2\tau, \infty)$$
	and, for all $n\in\mathbb{N}$, there exist indexes $i(n)$ and $k(n)$ for which 
	\begin{equation}\label{dt}
		d_x(t)=\lvert x_{i(n)}(t)-x_{k(n)}(t)\rvert,
	\end{equation}
	for all $t\in I_n$.
        Let us fix some $n\in \mathbb{N}$ and let $i:=i(n)$, $k:=k(n)$ be the indexes that satisfy condition \eqref{dt}. For $t\in I_n$ we have
	$$\begin{array}{l}
		\vspace{0.3cm}\displaystyle{\frac{1}{2}\frac{\d}{\d t}d_x(t)^2=\frac{1}{2}\frac{\d}{\d t}\lvert x_i(t)-x_k(t)\rvert^2=\langle x_i(t)-x_k(t),\dot{x}_i(t)-\dot{x}_k(t)\rangle}\\
		\vspace{0.3cm}\displaystyle{\hspace{0.3cm}= \left\langle
		     x_i(t)-x_k(t),\underset{j:j\neq i}{\sum}a_{ij}(t)(x_{j}(t-\tau)-x_{i}(t-\sigma))-\underset{j:j\neq k}{\sum}a_{kj}(t)(x_{j}(t-\tau)-x_{k}(t-\sigma)) \right\rangle}\\
		\vspace{0.3cm}\displaystyle{\hspace{0.3cm}=\underset{j:j\neq i}{\sum}a_{ij}(t)\langle x_i(t)-x_k(t),x_{j}(t-\tau)-x_{i}(t-\sigma)\rangle-\underset{j:j\neq k}{\sum}a_{kj}(t)\langle x_i(t)-x_k(t),x_{j}(t-\tau)-x_{k}(t-\sigma)\rangle}\\
		\displaystyle{\hspace{0.3cm}=: (\RI) + (\RII).}
	\end{array}$$
	Let us first estimate the term $(\RI)$. We write
	$$\begin{array}{l}
		\vspace{0.3cm}\displaystyle{(\RI)=\underset{j:j\neq i}{\sum}a_{ij}(t)\langle x_i(t)-x_k(t),x_{j}(t-\tau)-x_{j}(t)\rangle+\underset{j:j\neq i}{\sum}a_{ij}(t)\langle x_i(t)-x_k(t),x_{j}(t)-x_{i}(t)\rangle}\\
		\displaystyle{\hspace{1cm}+\underset{j:j\neq i}{\sum}a_{ij}(t)\langle x_i(t)-x_k(t),x_{i}(t)-x_{i}(t-\sigma)\rangle =: (\RI_1) + (\RI_2) + (\RI_3).}
	\end{array}$$ 
	The first term is estimated with the Cauchy-Schwarz inequality,
	$$(\RI_1) \leq \underset{j:j\neq i}{\sum}a_{ij}(t)\lvert x_i(t)-x_k(t)\rvert\lvert x_{j}(t-\tau)-x_{j}(t)\rvert=d_x(t)\underset{j:j\neq i}{\sum}a_{ij}(t)\lvert x_{j}(t-\tau)-x_{j}(t)\rvert.$$
	Now, since $t-\tau\geq 0$, we write
	$$\begin{array}{l}
		\vspace{0.3cm}\displaystyle{\lvert x_j(t-\tau)-x_j(t)\rvert
		   \leq \int_{t-\tau}^{t}\lvert\dot{x}_j(s) \rvert \,\d s\leq \int_{t-\tau}^{t}\max_{l\in[N]}\lvert \dot{x}_l(s)\rvert \,\d s.}
	\end{array}$$
	Since $s\geq t-\tau\geq \tau$, we can apply estimate \eqref{estonderiv} to get
	$$\lvert x_j(t-\tau)-x_j(t)\rvert\leq \int_{t-\tau}^{t}\left(d_x(s-\tau)+\int_{s-\tau}^{s}\max_{l\in[N]}\lvert \dot{x}_l(r)\rvert \,\d r\right) \d s.$$
	Hence, using \eqref{bound:a},
	$$\begin{array}{l}
		\vspace{0.3cm}\displaystyle{(\RI_1)\leq \left(d_x(t)\int_{t-\tau}^{t}d_x(s-\tau) \,\d s
		   + d_x(t)\int_{t-\tau}^{t}\int_{s-\tau}^{s}\max_{l\in[N]}\lvert \dot{x}_l(r)\rvert \,\d r\d s\right)\underset{j:j\neq i}{\sum}a_{ij}(t)}\\
		\displaystyle{\hspace{2cm}\leq d_x(t)\int_{t-\tau}^{t}d_x(s-\tau) \,\d s + d_x(t)\int_{t-\tau}^{t}\int_{s-\tau}^{s}\max_{l\in[N]}\lvert \dot{x}_l(r)\rvert \,\d r\d s.}
	\end{array}$$
	Similarly, we have
	$$(\RI_3)\leq d_x(t)\lvert x_i(t)-x_i(t-\sigma)\rvert\leq d_x(t)\int_{t-\sigma}^{t}\lvert \dot{x}_i(s)\rvert \,\d s.$$
	Then, since $s\geq t-\sigma\geq 2\tau-\sigma\geq \tau$, \eqref{estonderiv} yields
	$$(\RI_3) \leq d_x(t)\int_{t-\sigma}^{t}d_x(s-\tau) \,\d s+d_x(t)\int_{t-\sigma}^{t}\int_{s-\tau}^{s}\max_{l\in[N]}\lvert \dot{x}_l(r)\rvert \,\d r \d s.$$
	To estimate the term $(\RI_2)$, we use the Cauchy-Schwarz inequality and the identity $d_x(t)=\lvert x_i(t)-x_k(t)\rvert$, which gives
	$$\begin{array}{l}
		\vspace{0.3cm}\displaystyle{\langle x_i(t)-x_k(t),x_{j}(t)-x_{i}(t)\rangle=-\langle x_i(t)-x_k(t),x_{i}(t)-x_{k}(t)\rangle+\langle x_i(t)-x_k(t),x_{j}(t)-x_{k}(t)\rangle}\\
		\displaystyle{\hspace{0.3cm}\leq -d_x(t)^2+\lvert x_i(t)-x_k(t)\rvert\lvert x_j(t)-x_k(t)\rvert = d_x(t) (-d_x(t)+\lvert x_j(t)-x_k(t)\rvert)\leq 0,}
	\end{array}$$
	where we used $\lvert x_j(t)-x_k(t)\rvert\leq d_x(t)$ in the last inequality. Therefore,
	$$(\RI_2)\leq 
	   \underline{a}(t) \underset{j:j\neq i}{\sum} \langle x_i(t)-x_k(t),x_{j}(t)-x_{i}(t)\rangle.$$
	Combining the above estimates, we arrive at
	\begin{equation}\label{estI}
		\begin{array}{l}
			\vspace{0.3cm}\displaystyle{(\RI) \leq d_x(t)\left(\int_{t-\tau}^{t}d_x(s-\tau) \,\d s
			    + \int_{t-\sigma}^{t}d_x(s-\tau) \,\d s+\int_{t-\tau}^{t}\int_{s-\tau}^{s}\max_{l\in[N]}\lvert \dot{x}_l(r)\rvert \,\d r\d s\right.}\\
			\displaystyle{\hspace{2cm}\left.+\int_{t-\sigma}^{t}\int_{s-\sigma}^{s}\max_{l\in[N]}\lvert \dot{x}_l(r)\rvert \,d r\d s\right)
			  + \underline{a}(t) \underset{j:j\neq i}{\sum} \langle x_i(t)-x_k(t),x_{j}(t)-x_{i}(t)\rangle.}
		\end{array}
	\end{equation}
	For the term $(\mbox{\Romannum{2}})$ we write
	$$\begin{array}{l}
		\vspace{0.3cm}\displaystyle{(\RII) =-\underset{j:j\neq k}{\sum}a_{kj}(t)\langle x_i(t)-x_k(t),x_{j}(t-\tau)-x_{j}(t)\rangle-\underset{j:j\neq k}{\sum}a_{kj}(t)\langle x_i(t)-x_k(t),x_{j}(t)-x_{k}(t)\rangle}\\
		\displaystyle{\hspace{1cm}-\underset{j:j\neq k}{\sum}a_{kj}(t)\langle x_i(t)-x_k(t),x_{k}(t)-x_{k}(t-\sigma)\rangle =: (\RII_1) + (\RII_2) + (\RII_3).}
	\end{array}$$
    Applying similar steps as before, we obtain 
$$(\RII_1) \leq d_x(t)\int_{t-\tau}^{t}d_x(s-\tau) \,\d s + d_x(t)\int_{t-\tau}^{t}\int_{s-\tau}^{s}\max_{l\in[N]}\lvert \dot{x}_l(r)\rvert \,\d r\d s,$$
$$(\RII_3) \leq d_x(t)\int_{t-\sigma}^{t}d_x(s-\tau) \,\d s + d_x(t)\int_{t-\tau}^{t}\int_{s-\tau}^{s}\max_{l\in[N]}\lvert \dot{x}_l(r)\rvert \,\d r\d s,$$
and
	$$(\RII_2) \leq -\underline{a}(t) \underset{j:j\neq k}{\sum} \langle x_i(t)-x_k(t),x_{j}(t)-x_{k}(t)\rangle,$$
where in the last estimate above we used that $\langle x_i(t)-x_k(t),x_{j}(t)-x_{k}(t)\rangle\geq 0$.
	Consequently,
	\begin{equation}\label{estII}
		\begin{array}{l}
		\vspace{0.3cm}\displaystyle{(\RII) \leq d_x(t)\left(\int_{t-\tau}^{t}d_x(s-\tau) \,\d s
		   +\int_{t-\sigma}^{t}d_x(s-\tau) \,\d s + \int_{t-\tau}^{t}\int_{s-\tau}^{s}\max_{l\in[N]}\lvert \dot{x}_l(r)\rvert \,\d r\d s\right.}\\
		\displaystyle{\hspace{2cm}\left.+\int_{t-\sigma}^{t}\int_{s-\tau}^{s}\max_{l\in[N]}\lvert \dot{x}_l(r)\rvert drds\right)
		   - \underline{a}(t) \underset{j:j\neq k}{\sum} \langle x_i(t)-x_k(t),x_{j}(t)-x_{k}(t)\rangle.}
	\end{array}
	\end{equation}
	Combining estimates \eqref{estI} and \eqref{estII}  gives
	$$\begin{array}{l}
		\vspace{0.3cm}\displaystyle{\frac{1}{2}\frac{\d}{\d t}d_x(t)^2\leq 
		   2 d_x(t)\left(\int_{t-\tau}^{t}d_x(s-\tau) \,\d s + \int_{t-\sigma}^{t}d_x(s-\tau) \,\d s \right.} \\
		  \vspace{0.3cm}\displaystyle{\hspace{3cm}\left.  + \int_{t-\tau}^{t}\int_{s-\tau}^{s}\max_{l\in[N]}\lvert \dot{x}_l(r)\rvert \,\d r\d s
		       + \int_{t-\sigma}^{t}\int_{s-\tau}^{s}\max_{l\in[N]}\lvert \dot{x}_l(r)\rvert \,\d r\d s\right)  }\\
		  \vspace{0.3cm}\displaystyle{\hspace{2cm}
		       + \,\underline{a}(t) \left( \underset{j:j\neq i}{\sum} \langle x_i(t)-x_k(t),x_{j}(t)-x_{i}(t)\rangle
		          - \underset{j:j\neq k}{\sum} \langle x_i(t)-x_k(t),x_{j}(t)-x_{k}(t)\rangle \right).}
	\end{array}$$
	Now,
	$$\begin{array}{l}
		\vspace{0.3cm}\displaystyle{\underset{j:j\neq i}{\sum} \langle x_i(t)-x_k(t),x_{j}(t)-x_{i}(t)\rangle - \underset{j:j\neq k}{\sum} \langle x_i(t)-x_k(t),x_{j}(t)-x_{k}(t)\rangle}\\
		\vspace{0.3cm}\displaystyle{\hspace{1cm} = \underset{j:j\neq i,k}{\sum} (\langle x_i(t)-x_k(t),x_{j}(t)-x_{i}(t)\rangle-\langle x_i(t)-x_k(t),x_{j}(t)-x_{k}(t)\rangle)}\\
		\vspace{0.3cm}\displaystyle{\hspace{2cm}-\lvert x_i(t)-x_k(t)\rvert^2-\lvert x_i(t)-x_k(t)\rvert^2}\\
		\vspace{0.3cm}\displaystyle{\hspace{1cm}=-(N-2)\lvert x_i(t)-x_k(t)\rvert^2-2\lvert x_i(t)-x_k(t)\rvert^2}\\
		\displaystyle{\hspace{1cm}=-N\lvert x_i(t)-x_k(t)\rvert^2 = - N d_x(t)^2.}
	\end{array}$$
	Consequently,
	$$\begin{array}{l}
		\vspace{0.3cm}\displaystyle{\frac{1}{2}\frac{\d}{\d t}d_x(t)^2\leq d_x(t) \left(2\int_{t-\tau}^{t}d_x(s-\tau) \,\d s+ 2 \int_{t-\sigma}^{t}d_x(s-\tau) \,\d s
		   +2\int_{t-\tau}^{t}\int_{s-\tau}^{s}\max_{l\in[N]}\lvert \dot{x}_l(r)\rvert \,\d r\d s\right.}\\
		\displaystyle{\hspace{3cm}\left. + 2\int_{t-\sigma}^{t}\int_{s-\tau}^{s}\max_{l\in[N]}\lvert \dot{x}_l(r)\rvert \,\d r\d s- N \underline{a}(t)d_x(t)\right),}
	\end{array}$$
	which directly implies that \eqref{estwithextraterm} 
	holds for almost all $t\in I_n$. Repeating the above argument for every interval $I_n$, $n\in\N$,
	we deduce that \eqref{estwithextraterm} holds for almost all $t\in (2\tau,\infty)$.
\end{proof}

\begin{cor}\label{cor:aux}
	For almost all $t> 2\tau$ it holds
	\begin{equation}\label{estextraterm2}
		\begin{array}{l}
		\vspace{0.3cm}\displaystyle{\frac{\d}{\d t}d_x(t)\leq 4\int_{t-\tau}^{t}d_x(s-\tau) \,\d s + 4\tau\int_{t-2\tau}^{t}\max_{l\in[N]}\lvert \dot{x}_l(r)\rvert \,\d r- N \underline{a}(t)d_x(t),}
		\end{array}
	\end{equation}
	where $\underline{a}=\underline{a}(t)$ is defined in \eqref{underlinea}.
\end{cor}

\begin{proof}
	Due to the nonnegativity of the diameter $d_x=d_x(t)$ and the assumption $\tau\geq\sigma$, \eqref{estwithextraterm} readily gives
	$$\begin{array}{l}
		\vspace{0.3cm}\displaystyle{\frac{\d}{\d t}d_x(t)\leq 4\int_{t-\tau}^{t}d_x(s-\tau) \,\d s
		     + 2\int_{t-\tau}^{t}\int_{s-\tau}^{s}\max_{l\in[N]}\lvert \dot{x}_l(r)\rvert \,\d r\d s}\\
		\displaystyle{\hspace{3cm} +2\int_{t-\sigma}^{t}\int_{s-\tau}^{s}\max_{l\in[N]}\lvert \dot{x}_l(r)\rvert \,\d r\d s - N \underline{a}(t)d_x(t),}
	\end{array}$$
	for almost all $t> 2\tau$.
	Moreover, for all $s\in [t-\tau,t]$ we have $[s-\tau,s]\subset [t-2\tau,t]$, so that
	$$\begin{array}{l}
		\vspace{0.3cm}\displaystyle{2\int_{t-\tau}^{t}\int_{s-\tau}^{s}\max_{l\in[N]}\lvert \dot{x}_l(r)\rvert \,\d r\d s
		   \leq 2\int_{t-\tau}^{t}\int_{t-2\tau}^{t}\max_{l\in[N]}\lvert \dot{x}_l(r)\rvert \,\d r\d s}\\
		\displaystyle{\hspace{5cm}=2\tau\int_{t-2\tau}^{t}\max_{l\in[N]}\lvert \dot{x}_l(r)\rvert \,\d r.}
	\end{array}$$
	Similarly, since for all $s\in [t-\sigma,t]$ we have $[s-\tau,s]\subset[t-\tau-\sigma,t]\subset[t-2\tau,t]$, we find
	$$2\int_{t-\sigma}^{t}\int_{s-\tau}^{s}\max_{l\in[N]}\lvert \dot{x}_l(r)\rvert \,\d r\d s\leq
	   2\sigma\int_{t-2\tau}^{t}\max_{l\in[N]}\lvert \dot{x}_l(r)\rvert \,\d r\leq 2\tau\int_{t-2\tau}^{t}\max_{l\in[N]}\lvert \dot{x}_l(r)\rvert \,\d r,$$
	   and \eqref{estextraterm2} follows.
\end{proof}

\subsection{The Lyapunov-Krasovskii functional}\label{subsec:Lyap}
For $\{x_i\}_{i\in[N]}$ a global solution to \eqref{hkp}--\eqref{bound:a}
and a constant $\beta>0$ we introduce the functional $F: \RR^+\rightarrow \RR^+$,
\(   \label{def:F}
	F(t) :=
		d_x(t)+\beta\int_{[t-2\tau]^+}^{t}e^{-(t-s)}\int_{s}^{t}\underset{i\in[N]}{\max}\lvert \dot{x}_i(r)\rvert \,\d r\d s \qquad\mbox{for } t\geq 0,
\)
where we use the notation $[t-2\tau]^+ := \max\{0, t-2\tau\}$.
By definition, $F$ is continuous and almost everywhere differentiable on $\RR^+$.

\begin{lem}\label{Fder}
	Assume that
	\(  \label{ass:beta}
	2\beta e^{-2\tau}-4\tau-\beta\geq 0.
	\)
	Then, for almost all $t> 2\tau$,
	\begin{equation}\label{noextraterm}
		\frac{\d}{\d t}F(t) \leq 4\int_{t-\tau}^{t}F(s-\tau)\d s- (N-1) \underline{a}(t)F(t)+\beta(1-e^{-2\tau})F(t-\tau),
	\end{equation}
	where $\underline{a}(\cdot)$ is defined in \eqref{underlinea}.
\end{lem}

\begin{proof}
	For almost all $t>2\tau$ we have 
	$$\begin{array}{l}
		\vspace{0.3cm}\displaystyle{\frac{\d}{\d t} F(t) = \frac{\d}{\d t} d_x(t) - \beta e^{-2\tau}\int_{t-2\tau}^{t}\max_{i\in[N]}\lvert \dot{x}_i(r)\rvert \,\d r-\beta\int_{t-2\tau}^{t}e^{-(t-s)}\int_{s}^{t}\max_{i\in[N]}\lvert \dot{x}_i(r)\rvert \,\d r\d s}\\
		\vspace{0.3cm}\displaystyle{\hspace{2cm}+\beta\max_{i\in[N]}\lvert \dot{x}_i(t)\rvert \int_{t-2\tau}^{t}e^{-(t-s)} \,\d s}\\
		\vspace{0.3cm}\displaystyle{\hspace{1.4cm}=\frac{\d}{\d t}d_x(t)-\beta e^{-2\tau}\int_{t-2\tau}^{t}\max_{i\in[N]}\lvert \dot{x}_i(r)\rvert \,\d r-\beta\int_{t-2\tau}^{t}e^{-(t-s)}\int_{s}^{t}\max_{i\in[N]}\lvert \dot{x}_i(r)\rvert \,\d r\d s}\\
		\displaystyle{\hspace{2cm}+\beta(1-e^{-2\tau}) \max_{i\in[N]}\lvert \dot{x}_i(t)\rvert.}
	\end{array} $$
	Using \eqref{estextraterm2} to estimate the term $ \frac{\d}{\d t} d_x(t)$ gives
	$$\begin{array}{l}
		\vspace{0.3cm}\displaystyle{\frac{\d}{\d t}F(t)\leq 4\int_{t-\tau}^{t}d_x(s-\tau)\,\d s + \left( 4\tau-\beta e^{-2\tau}\right) \int_{t-2\tau}^{t}\max_{l\in[N]}\lvert \dot{x}_l(r)\rvert \,\d r
		   - N\underline{a}(t)d_x(t) }\\
		\vspace{0.3cm}\displaystyle{\hspace{3cm}-\beta\int_{t-2\tau}^{t}e^{-(t-s)}\int_{s}^{t}\max_{i\in[N]}\lvert \dot{x}_i(r)\rvert \,\d r\d s
		   +\beta \left(1-e^{-2\tau} \right) \max_{i\in[N]}\lvert \dot{x}_i(t)\rvert.}
	\end{array}$$
	Applying \eqref{estonderiv} to the last term of the right-hand side, we obtain 
	$$\begin{array}{l}
		\vspace{0.3cm}\displaystyle{\frac{\d}{\d t} F(t)\leq 4\int_{t-\tau}^{t}d_x(s-\tau)\,\d s
		   + \left(4\tau - 2\beta e^{-2\tau} + \beta\right) \int_{t-2\tau}^{t}\max_{l\in[N]}\lvert \dot{x}_l(r)\rvert \,\d r - N \underline{a}(t)d_x(t)}\\
		\displaystyle{\hspace{1.5cm}-\beta\int_{t-2\tau}^{t}e^{-(t-s)}\int_{s}^{t}\max_{i\in[N]}\lvert \dot{x}_i(r)\rvert \,\d r\d s
		  +\beta \left(1-e^{-2\tau} \right) d_x(t-\tau),}
	\end{array}$$
	where we used the estimate $\int_{t-\tau}^{t-\sigma}\max_{i\in[N]} \lvert \dot{x}_i(r)\rvert \,\d r\leq \int_{t-2\tau}^{t}\max_{i\in[N]} \lvert \dot{x}_i(r)\rvert \,\d r$.
	Now, by definition \eqref{underlinea} and assumption \eqref{bound:a}, we have $(N-1)\underline{a}(t) \leq 1$. Therefore,
	$$\begin{array}{l}
		\vspace{0.3cm}\displaystyle{-N\underline{a}(t)d_x(t)-\beta\int_{t-2\tau}^{t}e^{-(t-s)}\int_{s}^{t}\max_{i\in[N]}\lvert \dot{x}_i(r)\rvert \,\d r\d s}\\
		\displaystyle{\hspace{2cm}\leq - (N-1) \underline{a}(t)\left(d_x(t)+\beta\int_{t-2\tau}^{t}e^{-(t-s)}\int_{s}^{t}\max_{i\in[N]}\lvert \dot{x}_i(r)\rvert \,\d r\d s\right)}\\
		  \vspace{0.3cm}\displaystyle{\hspace{2cm} = - (N-1)\underline{a}(t)F(t).}
	\end{array}$$
	Consequently,
	\[
	   \frac{\d}{\d t}F(t) &\leq&  4\int_{t-\tau}^{t}d_x(s-\tau)\,\d s + \left( 4\tau - 2\beta e^{-2\tau}+ \beta\right)\int_{t-2\tau}^{t}\max_{l\in[N]}\lvert \dot{x}_l(r)\rvert \,\d r \\
	   && - (N-1) \underline{a}(t)F(t)+\beta(1-e^{-2\tau})d_x(t-\tau).
	\]
	Recalling the assumption $4\tau - 2\beta e^{-2\tau}+ \beta \leq 0$, we have
	\[
 	   \frac{\d}{\d t} F(t) \leq 4\int_{t-\tau}^{t}d_x(s-\tau)\,\d s - (N-1)\underline{a}(t)F(t) + \beta(1-e^{-2\tau})d_x(t-\tau).
	\]
	We conclude by observing that $d_x(t) \leq F(t)$ for all $t\geq 0$.
\end{proof}

\begin{lem} \label{lem:est_xi_xj}
For all $i,j\in[N]$ and $t\geq 2\tau$ we have  
\(   \label{est_xi_xj}
   \left| x_j(t-\tau) - x_i(t-\sigma) \right| \leq \dx_x(t-\tau) +  \int_{t-\tau}^{t-\sigma} \dx_x(s-\tau) \, \d s
      + \beta^{-1} e^{2\tau} F(t-\sigma).
\)
\end{lem}

\begin{proof}
For all $i,j\in[N]$, we estimate
\begin{equation}\label{x_i_x_j}
	\begin{array}{l}
		\vspace{0.3cm}\displaystyle{\left| x_j(t-\tau) - x_i(t-\sigma) \right|\leq \left| x_j(t-\tau) - x_i(t-\tau) \right| +  \left| x_i(t-\tau) - x_i(t-\sigma) \right|}\\
		\displaystyle{\hspace{2cm}\leq\dx_x(t-\tau) +  \int_{t-\tau}^{t-\sigma} \left| \dot x_i(s) \right| \d s.}
	\end{array}
\end{equation}
Estimate \eqref{estonderiv} gives
\[
    \left| \dot x_i(s) \right|  \leq \dx_x(s-\tau) + \int_{s-\tau}^{s-\sigma} \max_{k\in[N]} \left| \dot x_k(r) \right| \d r,
\]
so that
\[
   \int_{t-\tau}^{t-\sigma} \left| \dot x_i(s) \right| \d s \leq  \int_{t-\tau}^{t-\sigma} \dx_x(s-\tau) \, \d s
      + \int_{t-\tau}^{t-\sigma} \int_{s-\tau}^{s-\sigma} \max_{k\in[N]} \left| \dot x_k(r) \right| \d r\d s     
\]
Using the change of variable $\tilde s := s-\tau$ in the second integral, we obtain
$$
    \begin{array}{l}
    	\vspace{0.3cm}\displaystyle{\int_{t-\tau}^{t-\sigma} \int_{s-\tau}^{s-\sigma} \max_{k\in[N]} \left| \dot x_k(r) \right| \d r\d s =
    		\int_{t-2\tau}^{t-\tau-\sigma} \int_{\tilde s}^{\tilde s+\tau-\sigma} \max_{k\in[N]} \left| \dot x_k(r) \right| \d r\d \tilde s}\\
    		\displaystyle{\hspace{2cm}\leq \int_{t-2\tau}^{t-\tau-\sigma} \int_{\tilde s}^{t-2\sigma} \max_{k\in[N]} \left| \dot x_k(r) \right| \d r\d \tilde s\leq \int_{t-2\tau}^{t-\sigma} \int_{\tilde s}^{t-\sigma} \max_{k\in[N]} \left| \dot x_k(r) \right| \d r\d \tilde s.}
    \end{array}
$$
Therefore,
\[
   \int_{t-\tau}^{t-\sigma} \left| \dot x_i(s) \right| \d s \leq  \int_{t-\tau}^{t-\sigma} \dx_x(s-\tau) \, \d s
     + \int_{t-2\tau}^{t-\sigma} \int_{s}^{t-\sigma} \max_{k\in[N]} \left| \dot x_k(r) \right| \d r\d s.
\]
Now, definition \eqref{def:F} gives
\[
   F(t-\sigma) &\geq&  \beta\int_{[t-\sigma-2\tau]^+}^{t-\sigma}e^{-(t-\sigma-s)}\int_{s}^{t-\sigma}\max_{i\in[N]}\lvert \dot{x}_i(r)\rvert \d r\d s \\
   &\geq& 
   \beta e^{-2\tau} \int_{t-2\tau}^{t-\sigma} \int_{s}^{t-\sigma}\max_{i\in[N]}\lvert \dot{x}_i(r)\rvert \d r\d s,
\]
where we used the fact that for $t\geq 2\tau$ we have $[t-\sigma-2\tau]^+ \leq t-2\tau$.
Consequently,
\[
   \int_{t-\tau}^{t-\sigma} \left| \dot x_i(s) \right| \d s &\leq&  \int_{t-\tau}^{t-\sigma} \dx_x(s-\tau) \, \d s
      + \beta^{-1} e^{2\tau} F(t-\sigma).
\]
Inserting the above estimate into \eqref{x_i_x_j} gives the desired claim \eqref{est_xi_xj}.
\end{proof}

Finally, we estimate the values of the Lyapunov-Krasovskii functional \eqref{def:F} in the time interval
$[0,2\tau]$. 
For this purpose we define, for $K\in\N$,
\(   \label{def:DD}
   \DD^K_x := \max_{i,j \in [N]} \max_{t\in [-\tau,K\sigma]} \max_{s\in [(K-1)\sigma,K\sigma]}\lvert x_j(t) - x_i(s) \rvert.
\)
Moreover, we recall the definition \eqref{def:DD0} of $\DD^0_x$,
\[
   \DD^0_x := \max_{i,j \in [N]} \max_{s, t\in [-\tau,0]}\lvert x^0_i(s) - x^0_j(t) \rvert.
\]
	
	\begin{lem}
	We have, for all $K\in\N$, the recurrence
	\(   \label{appB:recurrence}
	     \DD^K_x \leq \DD^0_x + (1+\sigma) \DD^{K-1}_x + \sigma \sum_{k=0}^{K-1} \DD^{k}_x.
	\)
	\end{lem}
	
	\begin{proof} 
	Let us recall that $\sigma\leq \tau$.
	For $K\in\N$ we define
	\(   \label{def:M}
	   \MM^K := \max_{i\in[N]} \max_{t\in [(K-1)\sigma,K\sigma]}\lvert \dot x_i(t) \rvert,
	\)
	and claim that $\MM^K \leq \DD^{K-1}_x$, for all $K\in\N$.
	Indeed, we observe that for any $t\in [0, \sigma]$ and any $i\in[N]$ we have from \eqref{hkp}, \eqref{bound:a} and \eqref{def:DD0},
	\(    \label{eq:appB1}
             \lvert \dot x_{i}(t) \rvert \leq \underset{j:j\neq i}{\sum} a_{ij}(t) \lvert x_{j}(t-\tau)-x_{i}(t-\sigma) \rvert \leq \DD^0,
	\)
	from which $\MM^1 \leq \DD^0_x$ immediately follows.
	Moreover, let us fix any $K\in\N$, $K\geq 2$.
	For any $t\in [(K-1)\sigma,K\sigma]$ we have
	$t-\tau \in [-\tau, (K-1)\sigma]$ and $t-\sigma \in [(K-2)\sigma, (K-1)\sigma]$, so that for any $i\in[N]$,
	\[
             \lvert \dot x_{i}(t) \rvert \leq \underset{j:j\neq i}{\sum} a_{ij}(t) \lvert x_{j}(t-\tau)-x_{i}(t-\sigma) \rvert \leq \DD^{K-1}_x,
	\]
	which gives $\MM^K \leq \DD^{K-1}_x$.
		
	Now, for any $i,j \in [N]$ and $t\in [-\tau,K\sigma]$, $s\in [(K-1)\sigma,K\sigma]$ we write
	\(  \label{AppB:expansion}
	   \lvert x_j(t) - x_i(s) \rvert \leq \lvert x_j(t) - x_j(0) \rvert + \lvert x_j(0) - x_i((K-1)\sigma) \rvert + \lvert x_i((K-1)\sigma) - x_i(s) \rvert.
	\)
	If $t\leq 0$, then we readily have $\lvert x_j(t) - x_j(0) \rvert \leq \DD^0_x$. On the other hand, if $t>0$, using \eqref{def:M} we estimate 
	\[
	    \lvert x_j(t) - x_j(0) \rvert \leq \int_0^t \lvert \dot x_j(r)\rvert \d r \leq \sigma \sum_{k=1}^K \MM^{k} \leq \sigma \sum_{k=0}^{K-1} \DD^{k}_x.
	\]
	We therefore write
	\[
	   \lvert x_j(t) - x_j(0) \rvert \leq \DD^0_x + \sigma \sum_{k=0}^{K-1} \DD^{k}_x.
	\]
	The second term of the right-hand side of \eqref{AppB:expansion} is estimated directly from the definition \eqref{def:DD} by $\DD^{K-1}_x$,
	while for the third term we again use \eqref{def:M},
	\[
	    \lvert x_i((K-1)\sigma) - x_i(s) \rvert \leq \int_{(K-1)\sigma}^s \lvert \dot x_i(u)\rvert \d u \leq \sigma \MM^{K} \leq \sigma \DD^{K-1}_x.
	\]
	In summary, we have
	\[
	      \lvert x_j(t) - x_i(s) \rvert \leq  \DD^0_x + (1+\sigma) \DD^{K-1}_x + \sigma \sum_{k=0}^{K-1} \DD^{k}_x.
	\]
	which readily gives \eqref{appB:recurrence}.
	\end{proof}
		
	For resolving the recurrence \eqref{appB:recurrence}, we define the comparison sequence $(Z^k_\sigma)_{K\in\N}$
	with $Z^0_\sigma:=1$ and
	\(   \label{recc:ZK}
	     Z^K_\sigma = Z^0_\sigma + (1+\sigma) Z^{K-1}_\sigma + \sigma \sum_{k=0}^{K-1} Z^{k}_\sigma, \qquad\mbox{for } K\in\N.
	\)
	A simple induction argument shows that $\DD^K_x \leq Z^K_\sigma \DD^0_x$, for all $K\in\N\cup\{0\}$.
	Moreover, by subtracting the expressions for $Z^k_\sigma$ and $Z^{K-1}$, we obtain the second-order recurence
	\[
	    Z^K_\sigma = (1+\sigma) \left( 2 Z^{K-1}_\sigma - Z^{K-2}_\sigma \right), \qquad\mbox{for } K\geq 2,
	\]
	with $Z^0_\sigma=1$, $Z^1_\sigma = 2(1+\sigma)$. This is easily resolved as
	\[
	    Z^K_\sigma = \frac{1}{2\sqrt{\sigma(1+\sigma)}} \left[ \left( (1+\sigma) + \sqrt{\sigma(1+\sigma)} \right)^{K+1} - \left( (1+\sigma) - \sqrt{\sigma(1+\sigma)} \right)^{K+1} \right].
	\]
	
We thus have the following result.

\begin{lem}\label{lem:LyapBound}
Let $K := \lceil 2\tau/\sigma \rceil$, i.e., the smallest integer such that $K\sigma \geq 2\tau$. 
Then,
\(  \label{eq:LyapBound}
    F(t) \leq \mathcal{Z}^K_\sigma \DD^0_x,  \qquad\mbox{ for } t\in[0,2\tau],
\)
with $\mathcal{Z}^K_\sigma$ given by \eqref{def:Z}.
\end{lem}

\begin{proof}
	Noticing the monotonicity $Z^{m-1}_\sigma \leq Z^m_\sigma$, for all $m\in \mathbb{N}$, we have
	\[
	   \max_{i\in[N]} \max_{t\in [0,2\tau]}\lvert \dot x_i(t) \rvert \leq \max_{k\in [K]} M^k \leq \max_{k\in [K]} \DD^{k-1}_x \leq \max_{k\in [K]}Z_\sigma^{k-1} \DD^0_x\leq Z^{K-1}_\sigma \DD^0_x.
	\]
	Moreover, by definition \eqref{def:DD} we have $$d_x(t)\leq \max_{k\in [K]}\DD^k_x \leq \max_{k\in [K]} Z^k_\sigma \DD^0_x\leq Z^K_\sigma \DD^0_x,$$
	for all $t\in [0,2\tau]$. We therefore estimate the Lyapunov-Krasovskii functional \eqref{def:F} for $t\in [0,2\tau]$ by
	\[
		F(t) &=& d_x(t) + \beta\int_{0}^{2\tau}e^{-(2\tau-s)}\int_{s}^{2\tau}\underset{i\in[N]}{\max}\lvert \dot{x}_i(r)\rvert \,\d r\d s \\
		   &\leq& Z^K_\sigma \DD^0_x + Z^{K-1}_\sigma \DD^0_x\, \beta\int_0^{2\tau} e^{-(2\tau-s)} (2\tau - s) \,\d s \\
		   &\leq& Z^K_\sigma \DD^0_x + Z^{K-1}_\sigma \DD^0_x\, \beta \left( 1 - (1+2\tau) e^{-2\tau} \right) \\
		   &=& \mathcal{Z}^K_\sigma \DD^0_x.
	\]
\end{proof}

\subsection{Proof of Theorem~\ref{thm:consensus}} \label{subsec:consensus}
\begin{proof}
By continuity of the influence function $\psi=\psi(s)$, condition \eqref{condition} implies that
for any sufficiently small $\varepsilon>0$ we have
\(  \label{cond_eps}
   4\tau+\beta(1-e^{-2\tau}) < \psi\left( \left(1 + \tau-\sigma + \beta^{-1}e^{2\tau} \right) (\mathcal{Z}^K_\sigma+ \varepsilon) \DD^0_x\right),
\)
with $\mathcal{Z}^K_\sigma$ given by \eqref{def:Z}.
Let us consider the set
\[
   \mathcal{S} := \left\{t > 0:\; F(s) < (\mathcal{Z}^K_\sigma + \varepsilon)\DD^0_x\text{ for all } s\in [0,t) \right\},
\]
Note that bound \eqref{eq:LyapBound} gives $F(t) \leq \mathcal{Z}^K_\sigma \DD^0_x$, for all $t\in[0,2\tau]$.
Consequently, $\mathcal{S}\neq \emptyset$ and we denote $T:=\sup\mathcal{S}$.
Due to the continuity of the Lyapunov functional \eqref{def:F}, we obviously have $T>2\tau$.

We claim that $T = +\infty$. By contradiction, let us assume that $T <+\infty$.
Then, again by continuity, we have $F(t)< (\mathcal{Z}^K_\sigma + \varepsilon)\DD^0_x$, for all $t\in [0,T)$, and
\(   \label{for_contradiction}
	F(T) = (\mathcal{Z}^K_\sigma + \varepsilon) \DD^0_x.
\)
Lemma~\ref{lem:est_xi_xj} 
gives, for all $i,j\in[N]$ and $t\in [2\tau, T]$,
\[
   \left| x_j(t-\tau) - x_i(t-\sigma) \right| &\leq& \dx_x(t-\tau) +  \int_{t-\tau}^{t-\sigma} \dx_x(s-\tau) \, \d s
      + \beta^{-1} e^{2\tau} F(t-\sigma)   \\
      &\leq&
      \dx_x(t-\tau) + (\tau-\sigma) \max_{s\in [t-\tau, t-\sigma]} \dx_x(s-\tau) + \beta^{-1} e^{2\tau} F(t-\sigma) \\
      &\leq&
      \left(1 + \tau-\sigma + \beta^{-1}e^{2\tau} \right) (\mathcal{Z}^K_\sigma + \varepsilon) \DD^0_x,
\]
where we used the bound $d_x(t) \leq F(t) \leq (\mathcal{Z}^K_\sigma + \varepsilon) \DD^0_x$, for all $t\in [0,T]$.
Therefore, with the monotonicity of the influence function $\psi=\psi(s)$, we have, for all $t\in [2\tau, T]$,
\[  
   \underline{a}(t) = \min_{i,j\in[N]} a_{ij}(t) \geq \frac{1}{N-1} \psi \left( \left(1 + \tau-\sigma + \beta^{-1}e^{2\tau} \right) (\mathcal{Z}^K_\sigma + \varepsilon) \DD^0_x\right).
\]
We set
\[
   \alpha:=4\tau,\qquad  \gamma := \beta\left(1 - e^{-2\tau}\right), 
   \qquad \eta:= \psi \left( \left(1 + \tau-\sigma + \beta^{-1}e^{2\tau} \right) (\mathcal{Z}^K_\sigma + \varepsilon) \DD^0_x\right).
\]
Then from \eqref{cond_eps} we have $\alpha+\gamma<\eta$.
Moreover, \eqref{noextraterm} gives
\[
    \frac{\d}{\d t}F(t)\leq \frac{\alpha}{\tau}\int_{t-\tau}^{t}F(s-\tau) \,\d s-\eta F(t)+\gamma F(t-\tau),\qquad\text{for almost all }  t\in (2\tau,T).
\]
An application of the Halanay inequality, Lemma~\ref{prop:Halanay}, gives then
\[
   F(t) \leq \left( \max_{s\in [0, 2\tau]} F(s) \right) e^{-\Gamma(t-2\tau)}\leq \mathcal{Z}^K_\sigma \DD^0_xe^{-\Gamma(t-2\tau)}, \qquad \text{for all } t\in [2\tau, T],
\]
where $\Gamma>0$ is the unique solution of \eqref{Gamma}. In particular, we have
\[
   F(T) \leq \mathcal{Z}^K_\sigma \DD^0_x e^{-\Gamma(T-2\tau)} < \mathcal{Z}^K_\sigma \DD^0_x,
\]
which is a contradiction to \eqref{for_contradiction}.
We thus conclude that $T=+\infty$. The limit passage $\varepsilon\to 0+$ gives then $F(t)\leq \mathcal{Z}^K_\sigma \DD^0_x $, for all $t\geq 0$. Thus, repeating the above argument, due to \eqref{condition} we obtain by the Halanay Lemma~\ref{prop:Halanay}
\[
   d_x(t) \leq F(t) \leq \mathcal{Z}^K_\sigma \DD^0_x e^{-\Gamma(t-2\tau)},  \qquad\text{for all } t\geq 2\tau,
\]
where $\Gamma>0$ is the unique solution of \eqref{Gamma}. This concludes the proof.
\end{proof}

\section{Asymptotic flocking in system \eqref{CS1}--\eqref{CS:incond}}\label{sec:flocking}
\setcounter{equation}{0}

Throughout this section let $\{x_i,v_i\}_{i\in[N]}$ be a global solution of \eqref{CS1}--\eqref{CS2} with the initial datum \eqref{CS:incond}.
For a constant $\beta>0$ we introduce the Lyapunov-Krasovskii functional $G: \RR^+ \rightarrow \RR+$,
\(   \label{def:G}
	G(t) :=  
		d_v(t) + \beta\int_{[t-2\tau]^+}^{t}e^{-(t-s)}\int_{s}^{t}\underset{i\in[N]}{\max}\lvert \dot{v}_i(r)\rvert \,\d r\d s,
\)
with the velocity diameter $d_v=d_v(t)$ given by \eqref{def:dv}.

The following three lemmas are obtained by very slight modifications
of their counterparts established in Sections~\ref{subsec:aux} and~\ref{subsec:Lyap}.
Their proofs are obtained mostly by replacing agents' positions and position diameters
by agents' velocities and velocity diameters, respectively.
We therefore omit their proofs.

\begin{lem}\label{lem:Gder}
	Assume that
	\[
	   2\beta e^{-2\tau}-4\tau-\beta\geq 0.
	\]
	Then, for almost all $t> 2\tau$,
	\begin{equation}\label{G:noextratermCS}
		\frac{\d}{\d t} G(t) \leq 4\int_{t-\tau}^{t}G(s-\tau) \,\d s - (N-1) \underline{a}(t)G(t) + \beta \left(1-e^{-2\tau}\right) G(t-\tau),
	\end{equation}
	where $\underline{a}(\cdot)$ is defined in \eqref{underlinea}.
\end{lem}

\begin{lem}\label{lem:GBound}
Let $K := \lceil 2\tau/\sigma \rceil$, i.e., the smallest integer such that $K\sigma \geq 2\tau$. 
Then,
\(  \label{eq:GBound}
    G(t) \leq \mathcal{Z}^K_\sigma \DD^0_v \qquad\mbox{ for } t\in[0,2\tau],
\)
with $\mathcal{Z}^K_\sigma$ given by \eqref{def:Z} and $\DD^0_v$ given by \eqref{def:DDv0}.
\end{lem}

\begin{lem} \label{lem:est_vi_vj}
For all $i,j\in[N]$ and $t\geq 2\tau$, we have  
\(   \label{est_vi_vj}
	\left| v_j(t-\tau) - v_i(t-\sigma) \right| \leq d_v(t-\tau) +  \int_{t-\tau}^{t-\sigma} d_v(s-\tau) \, \d s
	+ \beta^{-1} e^{2\tau} G(t-\sigma).
\)
\end{lem}

We now establish the following bound on $\lvert x_j(t-\tau)-x_i(t-\sigma) \rvert$.

\begin{lem}\label{lem:CS:xi_xj}
For all $i,j\in[N]$ and $t\geq 2\tau$, we have  
\(  \label{CS:xi_xj}
		\begin{array}{l}
			\vspace{0.3cm}\displaystyle{\lvert x_j(t-\tau)-x_i(t-\sigma) \rvert \leq \lvert x_j(\tau)-x_i(2\tau-\sigma) \rvert} \\
			\displaystyle{\hspace{1cm} + \int_{2\tau}^{t}\left(G(s-\tau)+\beta^{-1} e^{2\tau}G(s-\sigma)+\int_{s-\tau}^{s-\sigma} G(r-\tau)\,\d r 
				\right) \,\d s.}
		\end{array}
\)
\end{lem}

\begin{proof}
For all $i,j\in[N]$ and $t> 2\tau$, we have with the Cauchy--Schwarz inequality,
\[
   \frac{1}{2}\frac{\d}{\d t}\lvert x_j(t-\tau)-x_i(t-\sigma) \rvert^2 &=& \langle x_j(t-\tau)-x_i(t-\sigma),v_j(t-\tau)-v_i(t-\sigma)\rangle  \\
	&\leq& \lvert x_j(t-\tau)-x_i(t-\sigma)\rvert \lvert  v_j(t-\tau)-v_i(t-\sigma)\rvert.
\]
Therefore, for almost all $t> 2\tau$,
\(    \label{eq:dxv}
   \frac{\d}{\d t}\lvert x_j(t-\tau)-x_i(t-\sigma) \rvert\leq  \lvert  v_j(t-\tau)-v_i(t-\sigma)\rvert,
\)
and with \eqref{est_vi_vj},
\[
    \frac{\d}{\d t} \lvert x_j(t-\tau)-x_i(t-\sigma) \rvert
       &\leq& d_v(t-\tau) +  \int_{t-\tau}^{t-\sigma} d_v(s-\tau) \, \d s + \beta^{-1} e^{2\tau} G(t-\sigma)  \\
       &\leq& G(t-\tau) +  \int_{t-\tau}^{t-\sigma} G(s-\tau) \, \d s + \beta^{-1} e^{2\tau} G(t-\sigma),
\]
where we used the fact that, by definition \eqref{def:G} of $G$, $d_v(s)\leq G(s)$, for all $s>0$.
Integrating the above estimate over the interval $(2\tau, t)$ we obtain \eqref{CS:xi_xj}.
\end{proof}

We now define, for $K\in\N$,
\(   \label{def:DDvK}
   \DD_v^K := \max_{i,j \in [N]} \max_{t\in [-\tau,K\sigma]} \max_{s\in [(K-1)\sigma,K\sigma]}\lvert v_j(t) - v_i(s) \rvert.
\)

\begin{lem}\label{lem:CS:bndxtau}
Let $K := \lceil 2\tau/\sigma \rceil$, i.e., the smallest integer such that $K\sigma \geq 2\tau$. 
Then,
\(   \label{CS:bndxtau}
   \lvert x_j(\tau)-x_i(2\tau-\sigma) \rvert \leq \DD_x^0 + \mathcal{W}^K_\sigma \DD_v^0,
\)
with $\mathcal{W}^K_\sigma$ given by \eqref{def:W}, $\DD_x^0$ by \eqref{def:DD0}, and $\DD^0_v$ defined in \eqref{def:DDv0}.
\end{lem}

\begin{proof}
We note that due to \eqref{CS:incomp}, estimate \eqref{eq:dxv} in fact holds for all $t>0$.
Integrating it over the time interval $(\sigma,2\tau)$ yields
\[
   \lvert x_j(\tau)-x_i(2\tau-\sigma) \rvert \leq  \lvert x_j(-\tau+\sigma) - x_i(0) \rvert + \int_0^{2\tau-\sigma} \lvert v_j(s-\tau+\sigma)-v_i(s) \rvert \,\d s.
\]
Since $-\tau+\sigma\in [-\tau,0]$,
definition \eqref{def:DD0} gives $\lvert x_j(-\tau+\sigma) - x_i(0) \rvert \leq \DD_x^0$.
Moreover, $2\tau-\sigma\leq (K-1)\sigma$. Therefore, 
\[
\begin{split}
	\lvert x_j(\tau)-x_i(2\tau-\sigma) \rvert &\leq\DD_x^0 + \int_0^{(K-1)\sigma} \lvert v_j(s-\tau+\sigma)-v_i(s) \rvert \,\d s\\
	&=\DD_x^0  + \sum_{k=1}^{K-1}\int_{(k-1)\sigma}^{k\sigma} \lvert v_j(s-\tau+\sigma)-v_i(s) \rvert \,\d s.
\end{split}
\]
For all $k\in[K]$ and for all $s\in [(k-1)\sigma,k\sigma]$, we have $ s-\tau+\sigma \in [-\tau, k\sigma]$. Hence, by definition \eqref{def:DDvK},
\[
   \lvert v_j(s-\tau+\sigma)-v_i(s) \rvert \leq \DD_v^k, \qquad\mbox{for all } s\in [(k-1)\sigma, k\sigma], 
\]
for all $k\in [K]$. Then, using $\DD_v^k \leq Z^k_\sigma \DD_v^0$, we have
\[
   \sum_{k=1}^{K-1}\int_{(k-1)\sigma}^{k\sigma} \lvert v_j(s-\tau+\sigma)-v_i(s) \rvert \,\d s \leq \sigma \DD_v^0 \sum_{k=1}^{K-1} Z^k_\sigma.
\]
Finally, we use the recursive formula \eqref{recc:ZK} to conclude
\[
   \sigma \sum_{k=1}^{K-1} Z^k_\sigma = \sigma \left(\sum_{k=0}^{K-1} Z^k_\sigma - 1 \right) = Z^K_\sigma - (1+ \sigma)\left( Z^{K-1}_\sigma + 1\right) = \mathcal{W}^K_\sigma.
\]
\end{proof}

\subsection{Proof of Theorem~\ref{thm:flocking}}\label{subsec:flocking}

\begin{proof}
Let us denote, for $t\geq 2\tau$,
\[
   D_x(t) := \max_{i,j\in [N]} \, \lvert x_j(t-\tau)-x_i(t-\sigma) \rvert.
\]
Moreover, we define the set
\[
   \mathcal{S} := \left\{t > 2\tau:\, D_x(s) < \DD^0_x + \mathcal{W}_\sigma^K \DD^0_v + \frac{e^{C\tau}}{C} \left(1 + \beta^{-1}e^{2\tau + C\sigma} + e^{C\tau} (\tau-\sigma)\right) \mathcal{Z}^K_\sigma \DD^0_v,
      \text{ for all } s\in [2\tau,t) \right\},
\]
with $C>0$ given by \eqref{CS:condition}, $\mathcal{Z}^K_\sigma$ defined in \eqref{def:Z} and $\mathcal{W}^K$ in \eqref{def:W}.
Lemma~\ref{lem:CS:bndxtau} gives
\(  \label{Dx2tau}
   D_x(2\tau) \leq \DD^0_x + \mathcal{W}_\sigma^K \DD^0_v.
\)
Since $\DD^0_v > 0$ (otherwise the system is already in the flocking state initially),
due to the continuity of the particle trajectories, we have $\mathcal{S}\neq \emptyset$.
We denote $T:=\sup\mathcal{S}$ and observe that $T>2\tau$.

We claim that $T = +\infty$. By contradiction, let us assume that $T <+\infty$.
Then, again by continuity, we have
\(   \label{CS:for_contradiction}
	D_x(T) = \DD^0_x + \mathcal{W}^K \DD^0_v + \frac{e^{C\tau}}{C} \left(1 + \beta^{-1}e^{2\tau + C\sigma} + e^{C\tau} (\tau-\sigma)\right) \mathcal{Z}^K_\sigma \DD^0_v.
\)
Let us denote
\[
   \eta := \psi \left( \DD^0_x + \mathcal{W}_\sigma^K \DD^0_v + \frac{e^{C\tau}}{C} \left(1 + \beta^{-1}e^{2\tau + C\sigma} + e^{C\tau} (\tau-\sigma)\right) \mathcal{Z}^K_\sigma \DD^0_v \right).
\]
Then, due to the monotonicity of the influence function $\psi=\psi(s)$, we have for all $t\in [2\tau, T]$,
\[  
   \underline{a}(t) = \min_{i,j\in[N]} a_{ij}(t) \geq \frac{1}{N-1} \eta.
\]
Lemma~\ref{lem:Gder} gives then
\[
    \frac{\d}{\d t} G(t)\leq \frac{\alpha}{\tau}\int_{t-\tau}^{t}G(s-\tau) \,\d s-\eta G(t)+\gamma G(t-\tau),\qquad\text{for almost all }  t\in (2\tau,T),
\]
with $\alpha:=4\tau$ and $\gamma:=\beta\left(1 - e^{-2\tau}\right)$.
Note that assumption \eqref{CS:condition} reads
\[
         e^{C\tau}(4\tau e^{C\tau} + \beta(1-e^{-2\tau})) + C \leq \eta,
\]
and this readily implies $\alpha+\gamma = 4\tau + \beta\left(1 - e^{-2\tau}\right) < \eta$.
We may therefore apply the Halanay inequality, Lemma~\ref{prop:Halanay}, to obtain
\[
   G(t) \leq \left( \max_{s\in [0, 2\tau]} G(s) \right) e^{-\Gamma(t-2\tau)}\leq \mathcal{Z}^K_\sigma \DD^0_v\, e^{-\Gamma(t-2\tau)}, \qquad \text{for all } t\in [2\tau, T],
\]
where we used the bound provided by Lemma~\ref{lem:GBound} for the second inequality. Note that, due to \eqref{eq:GBound}, the above inequality holds true also for $t\in[0,2\tau]$.
\\Now, the constant $\Gamma>0$ is given by \eqref{Gamma}, therefore,
\[
   e^{C\tau}(4\tau e^{C\tau} + \beta(1-e^{-2\tau})) + C \leq \eta =
   e^{\Gamma\tau}(4\tau e^{\Gamma\tau} + \beta(1-e^{-2\tau})) + \Gamma,
\]
and, by monotonicity, $C\leq\Gamma$. Consequently,
\(   \label{Gt}
   G(t) \leq \mathcal{Z}^K_\sigma \DD^0_v\, e^{-C(t-2\tau)}, \qquad \text{for all } t\in [2\tau, T].
\)
Now, Lemma~\ref{lem:CS:xi_xj} gives, for all $t\in [2\tau, T]$,
\[
   D_x(t) \leq D_x(2\tau) + \int_{2\tau}^{t}\left(G(s-\tau)+\beta^{-1} e^{2\tau}G(s-\sigma)+\int_{s-\tau}^{s-\sigma} G(r-\tau)\,\d r  \right) \,\d s.
\]
Using \eqref{Gt}, we estimate the integral term from above by
\[
   && \left[ \frac{e^{C\tau}}{C} \left(1 - e^{-C(t-2\tau)}\right) + \frac{\beta^{-1} e^{2\tau}}{C} \left( e^{C\sigma} - e^{-C(t-2\tau-\sigma)}\right) \right. \\
     && \quad \left. + \frac{e^{2C\tau}}{C^2} \left( 1 - e^{-C(t-2\tau)} + e^{-C(t-\tau-\sigma)} - e^{-C(\tau-\sigma)} \right) \right] \mathcal{Z}^K_\sigma \DD^0_v \\
     && \quad <
   \frac{e^{C\tau}}{C} \left(1 + \beta^{-1}e^{2\tau + C\sigma} + e^{C\tau} (\tau-\sigma) \right) \mathcal{Z}^K_\sigma \DD^0_v,
\]
where we used the fact that $\sigma\leq\tau$ and $1 - e^{-C(\tau-\sigma)} \leq C(\tau-\sigma)$.
Therefore, with \eqref{Dx2tau} we have
\[
   D_x(T) < \DD^0_x + \mathcal{W}_\sigma^K \DD^0_v + \frac{e^{C\tau}}{C} \left(1 + \beta^{-1}e^{2\tau + C\sigma} + e^{C\tau} (\tau-\sigma)\right) \mathcal{Z}^K_\sigma \DD^0_v,
\]
which is a contradiction to \eqref{CS:for_contradiction}.
We conclude that $T=+\infty$. Consequently, repeating the above argument and applying the Halanay Lemma \ref{prop:Halanay}, we get, using that $d_v(t)\leq G(t)$, for all $t\geq 0$,
\begin{equation}\label{decaydv}
	d_v(t) \leq G(t) \leq \mathcal{Z}^K_\sigma \DD^0_v\, e^{-C(t-2\tau)}, \qquad\mbox{for all } t \geq 2\tau.
\end{equation}
Hence, since one can easily show that $$d_x(t)\leq d_x(0)+\int_{0}^{t}d_v(s)\d s,\quad \forall t\geq 0,$$
using \eqref{decaydv} we conclude that
\[
   \sup_{t \geq 0} d_x(t) \leq d_x(0) + \int_0^{+\infty} d_v(s) \,\d s \leq d_x(0) +\frac{e^{2C\tau}}{C}\mathcal{Z}^K_\sigma \DD^0_v< +\infty.
\]
\end{proof}

	\bigskip
\noindent {\bf Acknowledgements.} {\small E. Continelli and C. Pignotti are members of Gruppo Nazionale per l’Analisi Matematica,
	la Probabilità e le loro Applicazioni (GNAMPA) of the Istituto Nazionale di
	Alta Matematica (INdAM). E. Continelli is supported by PRIN 2022  (2022W58BJ5) {\it PDEs and optimal control methods in mean field games, population dynamics and multi-agent models}. C. Pignotti is partially supported by PRIN 2022  (2022238YY5) {\it Optimal control problems: analysis, approximation, and applications} and by	PRIN-PNRR 2022 (P20225SP98) {\it Some mathematical approaches to climate change and its impacts}.}
\begin{appendices}

	\section{}\label{appA}
	\setcounter{equation}{0}
	\noindent
	We establish the following Gr\"{o}nwall--Halanay--type lemma \cite{Halanay}.
	It is a slight modification of the results found in \cite{ChoiH1, H:SIADS:21, H:JMAA:22}. 
	We give its proof here for readers' convenience.
	
	\begin{lem}\label{prop:Halanay}
		Let $\tau>0$ and the constants $\alpha,\gamma, \eta>0$ be such that $\alpha+\gamma<\eta$.
		Let $u:[0,+\infty)\rightarrow \mathbb{R}$ be a nonnegative continuous function,
		differentiable almost everywhere on $(2\tau,+\infty)$, that satisfies 
		\begin{equation}\label{eq:Halanay}
			\frac{\d}{\d t}u(t)\leq \frac{\alpha}{\tau}\int_{t-\tau}^{t}u(s-\tau)\,\d s+\gamma u(t-\tau)-\eta u(t),\qquad \text{for almost all }t>2\tau.
		\end{equation}
		Then,
		\begin{equation}\label{udecay}
			u(t)\leq \left(\max_{s\in [0,2\tau]}u(s)\right)e^{-\Gamma(t-2\tau)}, \qquad \text{for all } t\geq 2\tau,
		\end{equation}
		where $\Gamma\in (0,\eta)$ is the unique solution of the equation
		\begin{equation}\label{Gamma}
			e^{\Gamma\tau}(e^{\Gamma\tau}\alpha+\gamma) + \Gamma = \eta.
		\end{equation}

	\end{lem}
	\begin{oss}
		Let us note that as soon as the condition $\alpha+\gamma<\eta$ is verified, equation \eqref{Gamma} indeed has a unique solution $\Gamma\in (0,\eta)$.
		Namely, for the continuous function
		$$f(x):=e^{x\tau}(e^{x\tau}\alpha+\gamma)+x -\eta, \qquad x\geq 0,$$
		we have $f(0)=\alpha+\gamma-\eta<0$ and $f(\eta)=e^{\eta\tau}(e^{\eta\tau}\alpha+\gamma)>0$. Therefore, there exists $\Gamma\in(0,\eta)$ such that $f(\Gamma)=0$.
		Finally, since $f$ is strictly monotonic,
		the solution is unique.
	\end{oss}

	\begin{proof}[Proof of Lemma~\ref{prop:Halanay}.]
	Let us define
	\[
	   M := \max_{s\in [0,2\tau]} u(s), \qquad w(t) := M e^{-\Gamma(t-2\tau)}, \quad\text{for } t\geq 2\tau,
	\]
	with $\Gamma\in (0,\eta)$ the unique solution of \eqref{Gamma}.	We either have $u(2\tau) < M = w(2\tau)$, then there exists $T>2\tau$ such that
	\(  \label{Halanay:uw}
	    u(t) < w(t), \quad\text{for all } t\in (2\tau, T).
	\)
	Or we have $u(2\tau) = M$, then from \eqref{eq:Halanay} 
	\[
	    D^+ u(2\tau) \leq \frac{\alpha}{\tau}M\tau+ \gamma M -\eta u(2\tau) = M (\alpha+\gamma - \eta),
	\]
	where we used that $u(t)\leq M$ for all $t\in [0,2\tau]$, and $D^+ u(2\tau)$ denotes the upper (right-hand) Dini derivative of $u=u(t)$ at $t=2\tau$.
	Then, since $\dot w(2\tau) = -M\Gamma$ and $-\Gamma = e^{\Gamma\tau}(e^{\Gamma\tau}\alpha+\gamma) - \eta > \alpha+\gamma-\eta$,
	\[
	   D^+ u(2\tau) < -\Gamma M=\dot w(2\tau),
	\]
	and again there exists $T>2\tau$ such that \eqref{Halanay:uw} holds.
	\\Now, let us denote with $$S:=\{t>2\tau:u(s)<w(s),\text{ for all }s\in (2\tau,t)\},$$
	and $T:=\sup S$. Then, from \eqref{Halanay:uw} $S\neq\emptyset$ and $T>2\tau$. We are going to prove that $T=+\infty$.

	For contradiction, let us assume that there exists $T<+\infty$ such that \eqref{Halanay:uw} holds and
	\(   \label{Halanay:contra}
	    u(T) = w(T). 
	\)
	Then, necessarily $ D_-u(T) \geq \dot w(T)$, where $D_-  u(T)$ denotes the lower (left-hand) Dini derivative of $u=u(t)$ at $t=T$, since otherwise we would contradict the maximality of $T$. So, since $\dot w(T) = -\Gamma w(T)$, we have
	\(   \label{Halanay:dotuw}
	    D_- u(T) \geq - \Gamma w(T).
	\)
	On the other hand, \eqref{eq:Halanay} combined with \eqref{Halanay:contra} and with the fact that $u(t)\leq w(t)$, for all $t\in [0,T]$, gives
	\[
	   D_- u(T)\leq \frac{\alpha}{\tau}\int_{T-\tau}^{T} w(s-\tau)\,\d s + \gamma w(T-\tau)-\eta w(T).
	\]
	Using the identities
	\[
	    \int_{T-\tau}^{T} w(s-\tau)\,\d s = \frac{w(T)}{\Gamma} e^{\Gamma\tau} \left( e^{\Gamma\tau} - 1 \right), \qquad
	    w(T-\tau) = e^{\Gamma\tau} w(T), 
	\]
	we arrive at
	\[
	    D_- u(T) \leq \left[ \frac{\alpha}{\Gamma\tau} e^{\Gamma\tau} \left( e^{\Gamma\tau} - 1 \right) + \gamma e^{\Gamma\tau} - \eta \right] w(T).
	\]
	Using the inequality $\frac{e^{\Gamma\tau}-1}{\Gamma\tau} < e^{\Gamma\tau}$ and \eqref{Gamma}, we have
	\[
	   \frac{\alpha}{\Gamma\tau} e^{\Gamma\tau} \left( e^{\Gamma\tau} - 1 \right) + \gamma e^{\Gamma\tau} - \eta
	   < \alpha e^{2\Gamma\tau} + \gamma e^{\Gamma\tau} - \eta
	   = -\Gamma,
	\]
	so that
	\[
	    D_- u(T) < - \Gamma w(T),
	\]
	which is a contradiction to \eqref{Halanay:dotuw}.
	We therefore conclude that $T=+\infty$, namely $u(t) \leq w(t)$, for all $t\geq 2\tau$, which is \eqref{udecay}.
	\end{proof}
	
\end{appendices}


\end{document}